\documentclass{amsproc}

\usepackage{amsmath,amscd,amsfonts,amssymb,amsthm, multirow, newlfont,exscale,color,mathrsfs,stmaryrd, multicol,subdepth, enumitem}

\newtheorem{theorem}{Theorem}[section]
\newtheorem{lemma}[theorem]{Lemma}
\newtheorem{corollary}[theorem]{Corollary}
\newtheorem{proposition}[theorem]{Proposition}
\theoremstyle{definition}

\newtheorem{example}[theorem]{Example}

\theoremstyle{remark}

\newtheorem{discussion}[theorem]{Discussion}

\newcommand{\inc}{\subseteq}

\newcommand{\C}{\mathbb{C}}

\newcommand{\cB}{\mathcal{B}}

\newcommand{\cF}{\mathcal{F}}

\newcommand{\inj}{\hookrightarrow}

\newcommand{\cM}{\mathcal{M}}

\newcommand\dlim[1]{\varinjlim\,_{#1}\,}

\newcommand{\fA}{\mathfrak{A}}
\newcommand{\fB}{\mathfrak{B}}

\newcommand{\fh}{\mathfrak{h}}

\newcommand{\fI}{\mathfrak{I}}

\newcommand{\fm}{\mathfrak{m}}
\newcommand{\fn}{\mathfrak{n}}
\newcommand{\fM}{\mathfrak{M}}
\newcommand{\fN}{\mathfrak{N}}

\newcommand{\gri}{geometrically reduced and irreducible}

\newcommand{\infl}{intersection flat}
\newcommand{\ins}{\preceq}
\newcommand{\inns}{\precneqq}

\newcommand{\la}{\lambda}

\newcommand{\N}{\mathbb{N}}
\newcommand{\nx}{\prec_{\scriptstyle \mathrm{im}}}

\newcommand{\pep}{prime extension property}
\newcommand{\step}{stable prime extension property}

\newcommand{\ov}{\overline}

\newcommand{\rifi}{respectively, \infl\ for ideals}

\newcommand{\surj}{\twoheadrightarrow}

\newcommand{\Tor}{\textrm{Tor}}
\newcommand{\Z}{\mathbb{Z}}

\newcommand\vect[2]{#1_1,\,\ldots,\, #1_{#2}}

\def\vect#1#2{{#1}_1, \, \ldots, \, {#1}_{#2}}

\newcommand{\mx}{\begin{pmatrix}}
	\newcommand{\emx}{\end{pmatrix}}

\newcommand{\wh}[1]{\widehat{#1}}

\newcommand{\Ann}{\mathrm{Ann}}

\newcommand{\bs}[1]{\boldsymbol{#1}}

\numberwithin{equation}{section}

\begin{document}

\title[Extensions of primes and flatness]{Extensions of primes, flatness, and\\ intersection flatness}

\author{Melvin Hochster}
\address{Department of Mathematics, East Hall, 530 Church St., Ann Arbor, MI 48109--1043, USA}
\curraddr{}
\email{hochster@umich.edu}
\thanks{The first author was partially supported by  National Science Foundation grants
	DMS--1401384 and DMS--1902116.}

\author{Jack Jeffries}
\address{Matem\'aticas B\'asicas Group, Centro de Investigaci\'on en Matem\'aticas, Guanajuato, Gto. 36023, M\'exico}
\curraddr{}
\email{jeffries@cimat.mx}

\thanks{The second author was partially supported by National Science Foundation grant DMS--1606353}

%%\subjclass{Primary 13B02, 13B10, 13B40, 13C11}


\begin{abstract}
We study when $R \to S$ has the property that prime ideals of $R$ extend to prime ideals
or the unit ideal of $S$, and the situation where this property continues to hold after adjoining the same indeterminates
to both rings.  We  prove that if $R$ is reduced, every maximal ideal of $R$ contains only finitely many minimal primes of $R$, and prime 
ideals of $R[\vect Xn]$ extend to prime ideals of $S[\vect X n]$ for all $n$,  then $S$ is flat over $R$. We give a counterexample
to flatness over a reduced quasilocal ring $R$ with infinitely many minimal primes by constructing a non-flat $R$-module $M$ such that
$M = PM$ for every minimal prime $P$ of $R$.  We study the notion of intersection flatness and use it to prove that  in certain graded cases it suffices
to examine just one closed fiber to prove the \step.
\end{abstract}

\maketitle

\begin{center}
	\dedicatory{\emph{Dedicated to Roger and Sylvia Wiegand on\\ the occasion of their 150th birthday}}
\end{center}

\section{Introduction}\label{intro} %%Section 1

All rings in this paper are assumed commutative, associative, with multiplicative identity.  Although we were originally
motivated in studying the Noetherian case, some of our results hold in much greater generality.

We say that an  $R$-algebra $S$ or that the homomorphism $R \to S$ has the {\it prime extension property}
if for every prime $P$ of $R$,   $PS$ is prime in $S$ or the unit ideal of~$S$.  We say that the $R$-algebra $S$
or the homomorphism $R \to S$ has the {\it stable prime extension property} if for every finite set of indeterminates
$\vect X n$ over these rings,  $R[\vect X n] \to S[\vect X n]$ has the \pep.

One of our results, Theorem~\ref{thm:main},  asserts the following. 

\begin{theorem}\label{th:intro-main} Let $R$ be a reduced ring such that every maximal ideal contains only finitely many minimal primes; in particular, this holds if $R$ is reduced and locally Noetherian.
	If $R \to S$ has the \step, then $S$ is flat over~$R$.
\end{theorem}

In the case $R$ is a domain, this follows from a result of Picavet \cite[Theorem~3.7]{Pi}, with a different proof.  Note that the hypothesis that $R$ be
reduced is quite necessary, since $R \to R/\fN_R$,  where $\fN_R$ is the ideal of nilpotent elements of $R$, also
has  the \step. It is easy to see that the hypothesis of having the \step\ cannot be weakened to having the \pep\ in Theorem~\ref{th:intro-main}; for example, since the surjection  $\C\llbracket X,Y \rrbracket / (XY) \twoheadrightarrow \C \llbracket X \rrbracket$ sending $Y\mapsto 0$, satisfies the \pep, but is not flat. Note that in this example, the \pep\ is lost if we adjoin an indeterminate $Z$ to both rings, since $(X, Y-Z^2)$ yields a prime
in  $\C\llbracket X,Y\rrbracket [Z]/(XY)$ but not in $\C \llbracket X\rrbracket [Z]$. The hypothesis on the minimal primes of $R$ is also necessary: in Section~\ref{sec:example}, we construct an inclusion of quasilocal rings in which the source is reduced that satisfies the \step, but is not flat.  A more refined but more technical version of Theorem~\ref{th:intro-main} is given in Theorem~\ref{main2}.  Our results here are related
to results of the first author on radical ideals in~\cite{Ho}.

We note that there are significant examples where the \step\ holds:  in particular,  the following result from \S5, Theorem~\ref{gr-reg}  generalizes \cite[Corollary~2.9]{AH} (also   cf.~\cite[Theorem~12.1(viii)]{EGA-IV3}). This surprisingly enables one to deduce that the stable prime extension property holds by examining one closed fiber.

\begin{theorem}\label{th:int-gr-reg} If $K$ is an algebraically closed field, $S$ is a $\Z$-graded $K$-algebra (but we are not assuming that $S$ is Noetherian nor that $S_0 = K$)
	and $\vect Fn$ are positive degree forms of $S$ with coefficients in $K$ that form a regular sequence and generate a prime ideal $Q$ of $S$, then
	$K[\vect Fn] \to S$ has the \step. \end{theorem}

This and related results in \S5 make use of the notion of {\it intersection flatness}. Some basic properties of this notion are established in \S5, and applied to give sufficient criteria for the \step. The results in \S5 are inspired by the work of T.~Ananyan and the first author in \cite{AH}. Intersection flatness is also closely related to the notion of \emph{content} in the sense of Ohm and Rush \cite{OR,Ru} (see also \cite{ES1, ES2}), which we recall in \S5.

\section{Basic properties}\label{basic} %%Section 2

We collect some basic properties of the \pep\ and the \step. These notions were studied by Picavet \cite{Pi} under the names ``prime producing'' and ``universally prime producing.'' Many of the properties established in this section also appear in \cite[\S 2]{Pi}.

%% 2.1
\begin{proposition}\label{bas1}  Let $R \to S$ be a ring homomorphism, and let  $R_\lambda \to S_\lambda$ be a direct limit system of ring
	homomorphisms indexed by $\lambda$. 
	
	\begin{enumerate}[label=(\alph*)]
		\item\label{2.1a} If $R \to S$ and $S \to T$ both have the \pep\ (respectively, \step), then so does the composite map $R \to T$.
		\item\label{2.1b} $R \to S$ has the \pep\ if and only if for every surjection $R \surj D$, where $D$ is a domain,  $D \otimes_R S$ is a domain or zero.
		\item\label{2.1c} $R \to S$ has the \step\ if and only if for every map $R\to D$, where $D$ is a domain finitely generated over $R$, $D \otimes_R S$ is a domain or zero. 
		\item\label{2.1d} If $R_\lambda \to S_\lambda$ has the \pep\ (respectively, \step) for each $\lambda$, so does $\dlim{\lambda} (R_\lambda \to S_\lambda)$.
		\item\label{2.1e} $R \to S$ has the \step\ if and only if for every map $R \to D$,  where $D$ is a domain, $D \otimes_R S$ is a domain or zero.
		\item\label{2.1f}  $R \to S$ has the \step\ if and only if for every homomorphism $R \to R'$, $R' \to R' \otimes_R S$ has the \pep, and this
		holds if and only for every homomorphism $R \to R'$, $R' \to R' \otimes_R S$ has the \step.
	\end{enumerate} 
\end{proposition}
\begin{proof} Part \ref{2.1a} is clear, since $T/PT \cong T/(PS)T$, and \ref{2.1b} is a consequence of the fact that 
	$D \cong R/P$. 
	For \ref{2.1c}, given a domain $D$ finitely generated over $R$, take a surjection $R[\vect Xn]\surj P$. Since $R[\vect Xn]\to S[\vect Xn]$ has the \pep, $S\otimes_R D \cong S[\vect Xn] \otimes_{R[\vect Xn]} D$ is a domain or zero by \ref{2.1b}. Conversely, if $P\subseteq R[\vect Xn]$ is a prime for which $PS[\vect Xn]$ is neither prime nor the unit ideal, then setting $D=R[\vect Xn]/P$ yields $D\otimes_R S \cong S[\vect Xn]/PS[\vect Xn]$ which is neither a domain nor zero.
	For \ref{2.1d}, note that if $P$ is a prime of $\dlim \lambda R_\lambda$ and $P_\lambda$ is the contraction of
	$P$ to $R_\lambda$, then $P = \dlim \lambda P_\lambda$,  and the result for the \pep\ follows from the fact that a direct limit of rings each of which is
	a domain or zero is itself a domain or zero.  The result for \step\ is then a consequence of the fact that direct limit commutes with adjoining the variables.
	Part \ref{2.1e} is a consequence of \ref{2.1c} and \ref{2.1d}: given a map from $R$ to a domain~$D$, one may write $D= \dlim{\lambda} D_{\lambda}$ for a directed system of domains that are finitely generated over $R$, and $S\otimes_R D\cong \dlim{\lambda} (S\otimes_R D_{\lambda})$. Part \ref{2.1f} follows at once from the characterization in \ref{2.1e} and the isomorphisms $(R' \otimes_R S)\otimes_{R'} D \cong S\otimes_R D$.
\end{proof}

The name ``\step'' is partly motivated by statement~\ref{2.1f} above. We will use the characterizations 
of the \pep\ in part~\ref{2.1b} and the \step\ in part~\ref{2.1e} of the previous proposition repeatedly below. 

%% 2.2
\begin{proposition}\label{bas2} Let $R \to S$ be any ring homomorphism.
	\begin{enumerate}[label=(\alph*)]
		
		\item\label{2.2a} If $R \to S$ has the \pep\ or the \step\ and $W$ is a multiplicative system in $R$ (respectively, $V$ is a multiplicative
		system in $S$) then $W^{-1}R \to W^{-1}S$ (respectively,  $R \to V^{-1}S$) has the same property.  %% a
		
		\item\label{2.2b} If  $\fA$ is an ideal consisting of nilpotent elements of $R$, then $R \to S$ has the \pep\ (respectively, the \step) if and 
		only if  $R/\fA  \to S/\fA S$ has that property.  %% b
		
		\item\label{2.2c} The map $R/PR \to S/PS$ has the \pep\  (respectively, the \step) for every minimal prime $P$ of $R$ if and only if $R\to S$ has that property.
		
		\item\label{2.2d} If $S$ is a polynomial ring in any family of variables over $R$, then $R \to S$ has the \step.  %%c
		
		\item\label{2.2e} If $R$ is Noetherian, a formal power series ring in any family of variables over $R$ has the \pep. %%d
		
	\end{enumerate}
\end{proposition}
\begin{proof} For part~\ref{2.2a}, suppose that $R\to S$ has the \pep\ and $W^{-1}R \surj D$ for some domain $D$. Since every prime of $W^{-1}R$ is expanded from $R$, we can write $D\cong W^{-1}D'$ for some domain $D'$ with $R\surj D'$. Then, by Proposition~\ref{bas1}\ref{2.1b}, $S\otimes_R D'$ is either a domain or zero, so \[W^{-1}S\otimes_{W^{-1}R} D \cong W^{-1}S\otimes_{W^{-1}R} W^{-1}D' \cong W^{-1} (S\otimes_R D')\] is either a domain or zero, so $W^{-1}R \to W^{-1}S$ has the \pep. The argument for \step\ in the same, except replacing surjections to domains with general maps to domains, and using Proposition~\ref{bas1}\ref{2.1e}. The case with $V$ in place of $W$ is similar.
	
	Part \ref{2.2b} follows from the fact all maps from $R$  to a domain $D$ must have $\fA$ in their kernel, so $S\otimes_R D\cong S/\fA S \otimes_{R/\fA} D$; thus the condition of Proposition~\ref{bas1}\ref{2.1b} holds or fails simultaneously for the two given maps, and likewise for Proposition~\ref{bas1}\ref{2.1e}.
	
	For the forward implication of \ref{2.2c}, a map $R\to D$ to a domain $D$ must contain a minimal prime $P$ in the kernel, so $S\otimes_R D\cong S/P S \otimes_{R/ P} D$ for that prime $P$, and the conclusion follows from Proposition~\ref{bas1}\ref{2.1b} and~\ref{2.1e}. The reverse follows from Proposition~\ref{bas1}\ref{2.1f}.
	
	Parts \ref{2.2d} and \ref{2.2e}
	follow from the fact that $S/PS$ may be identified with the corresponding polynomial or power series ring over $R/P$,
	since for any ideal (respectively, finitely generated ideal) $I$, the expansion $IS$ is the same as the ideal of polynomials (respectively power series) all of whose coefficients are in $I$.  Moreover, in the polynomial case, the hypothesis continues to hold after adjoining
	indeterminates to both rings.    
\end{proof}

\section{Flatness}\label{sec-main} %%Section 3

We say that $T$ is a {\it geometrically reduced and irreducible} algebra over the field $K$ if for every field extension $K \inc L$, $L \otimes_K T$ is a domain.  
%% since the tensor productof two domains over an algebraically closed field is a domain. [IS A REFERENCE NEEDED? -- I don't think so.
If $D \inc L$ is a domain, 
then $D \otimes_K T \inc L \otimes_K T$.  Thus $K \to T$ is \gri\ if and only if it has the \step.

Observe that a direct limit of \gri \ algebras is again \gri, since tensor products commute with direct limits. If $T_0 \inc T$ is a $K$-subalgebra, then
$L\otimes_K T_0 \inc L \otimes_K T$, so that $T$ is \gri\ if and only if all $K$-subalgebras are \gri, which happens if and only if all finitely generated $K$-subalgebras are
\gri.  

If $P$ is a prime ideal of $R$, we let $\kappa_P = R_P/PR_P$,  which is canonically isomorphic with the field of fractions of $R/P$.
The {\it fiber} of $R \to S$ over a prime ideal $P$ of $R$ is $\kappa_P \otimes_R S$.

%% 3.1
\begin{proposition}\label{prop:step-gri} If $R \to S$ has \step, then for all primes $P$ of $R$,  the fiber $\kappa_P \otimes_R S$ is a \gri\ $\kappa_P$-algebra.  If, moreover,
	$R \to S$ is flat, the converse holds, i.e. if  $R \to S$ is flat, then $R \to S$ has the \step\ if and only 
	if $\kappa_P \otimes_R S$ is a \gri\ $\kappa_P$-algebra for all primes $P$ of $R$. \end{proposition}
\begin{proof} Assume that $R \to S$ has the \step.  If $L$ is any extension field of $\kappa_P$, we have a composite map $R \to \kappa_P \to L$,
	and $$L \otimes_{\kappa_P} (\kappa_P \otimes_R S) \cong L \otimes_R S$$ is a domain by Proposition~\ref{bas1}\ref{2.1e}. 
	
	Now assume that $R \to S$ is flat and that all fibers are \gri.  Suppose that we have a homomorphism $R \to D$, where $D$ is a domain, and
	let $P$ be the kernel. Let $L$ be the field of fractions of $D$.   Then $D \otimes_R S$ is flat over $D$. Hence, the elements of $D \smallsetminus \{0\}$ are nonzerodivisiors, 
	and to show that $D \otimes_R S$ is a domain it suffices to show that $L \otimes_R S$ is a domain. Since we have an injection $\kappa_P \inj L$,
	we have the identification $L \otimes_{\kappa_P} (\kappa_P \otimes_R S) \cong L \otimes_R S$, and this ring is a domain by the hypothesis that $\kappa_P \otimes_R S$ is \gri. \end{proof}

Our next goal is to show that under mild conditions on the reduced ring $R$, the condition that $R \to S$ has the \step\ forces the flatness of $S$ over $R$.  See
Theorem~\ref{thm:main} below. We need some preliminary results.

The result that intersecting two ideals commutes with extension from $R$ to $S$ is often stated for the case where $S$ is flat over $R$.  We prove
that the result holds under a much weaker assumption:  that for one of the ideals $\fA$,  $S/\fA S$ is flat over $R/\fA$.  In fact, we assume even less:

%%3.2
\begin{lemma}\label{cap}  If $R \to S$ is any ring homomorphism and $\fA,\, \fB \inc R$ are ideals such that $\Tor_1^{R/\fA}(R/(\fA +\fB),\,S/\fA S) = 0$, which
	holds, for example, if $S/\fA S$ is flat over $R/\fA$, then    $\fA S \cap \fB S = (\fA \cap \fB)S$. \smallskip
	Hence, if $\vect \fA h$ are ideals of $R$ such that $S/\fA_i S$ is flat over $R/\fA_i$ for $1 \leq i \leq h-1$,  then $\bigcap_i (\fA_i S) = (\bigcap_i \fA_i)S$.
\end{lemma} 
\begin{proof} We use an overline to indicate images of elements in $R/\fA$ or $S/\fA S$: context should make it clear which is meant. 
	Let $u \in \fA S \cap \fB S$. Let $\{ g_{\lambda} \ | \ \lambda\in \Lambda\}$ be a generating set for $\fB$, and take a free resolution
	\[ \cdots \to (R/\fA)^{\oplus \Gamma} \xrightarrow{d_1} (R/\fA)^{\oplus \Lambda} \xrightarrow{d_0} R/\fA \to 0  \]
	for $R/(\fA+\fB)$ as an $(R/\fA)$-module, where $d_0$ maps the generator indexed by $\lambda$ to $\overline{g_{\lambda}}$. We can write 
	\[  u = \sum_{j=1}^k g_{\lambda_j} s_{\lambda_j} \in \fA S\]  with $s_{\lambda_1},\dots,s_{\lambda_k} \in S$ and $\lambda_1,\dots,\lambda_k \in \Lambda$. Let $\bs{s}$ be the vector with coordinate ${s_{\lambda_j}}$ in index $\lambda_j$, and all other coordinates zero. The relation $\overline{u}=0$ above specifies an element in $\ker(S/\fA S \otimes_{R/\fA} d_0)$, namely, the vector $\overline{\bs{s}}$ with coordinate $\overline{s_{\lambda_j}}$ in index $\lambda_j$, and all other coordinates zero.

	The vanishing
	of $$\Tor_1^{R/\fA}(R/(\fA+\fB), \,S/\fA S)$$  implies that $\overline{\bs{s}}\in \mathrm{im}(S/\fA S \otimes_{R/\fA} d_1)$, i.e., 
	the vector $\overline{\bs{s}}$ of coefficents is an {$(S/\fA S)$-linear} combination of finitely many, say $h$, vectors $\overline{\bs{r_1}},\dots, \overline{\bs{r_h}}$ with coefficients in $R/\fA$ that each give relations on the  $\ov{g_{\lambda}}$  in $R/\fA$.
	We lift each $\overline{r_i}$ to a vector $r_i$ with coefficients in $R$ by choosing arbitrary preimages for the nonzero coordinates, and zero in the zero coordinates. Note that each $\overline{\bs{r_i}}$ and hence each $r_i$, has finite support, so there is a finite subset $\Lambda_0=\{\lambda_1,\dots,\lambda_w\}$ that contains the support of all of the vectors $\bs{s},\bs{r_1},\dots,\bs{r_h}$. Then, we have
	
	\begin{enumerate}
		\item\label{item-1-proof} $\bs{s}  \equiv \sum_{i=1}^h u_i \bs{r_i}$ modulo $\fA S$ (as vectors, coordinatewise) for some elements $u_i \in S$, and
		\item\label{item-2-proof}   for every $i$,  $\sum_{\lambda \in \Lambda} {\bs{r_i}}_{\lambda} g_{\lambda} = \sum_{\lambda \in \Lambda_0} {\bs{r_i}}_{\lambda} g_{\lambda} \in  \fA$, and hence $\fA  \cap \fB$,  because of the presence of the $g_{\lambda}$.
	\end{enumerate}
	
	\noindent	From (\ref{item-1-proof}) we have 
	
	\begin{enumerate}\setcounter{enumi}{2}
		\item\label{item-3-proof} for $\lambda\in \Lambda_0$,  $s_\lambda = \sum_{i=1}^h u_i {\bs{r_i}}_{\lambda} + f_\lambda$  for some $f_\lambda \in \fA S$.
	\end{enumerate}
	
	\noindent	Then
	\[ u = \sum_{\lambda\in \Lambda_0} g_{\lambda} s_{\lambda} =  \sum_{\lambda\in \Lambda_0} g_{\lambda} \left(\sum_{i=1}^h u_i {\bs{r_i}}_{\lambda} + f_\lambda\right) =\sum_{i=1}^h u_i \left( \sum_{\lambda\in \Lambda_0} {\bs{r_i}}_{\lambda} g_\lambda \right) + \sum_{\lambda\in \Lambda_0} f_{\lambda} g_{\lambda} \]
	and the result follows because the terms in the first sum on the right are in $(\fA \cap \fB)S$ by (\ref{item-2-proof}) and each  $f_\la g_\la \in (\fA  \fB)S$.
	The final statement follows by a straightforward induction on $h$.
\end{proof}

%%3.3
\begin{samepage}
	\begin{proposition}\label{ker}  Let $R \to S$ have the \step. 
		\begin{enumerate}[label=(\alph*)]
			\item\label{3.3a}  If $(R,\,\fm)\to (S, \fn)$ is a local homomorphism of quasilocal rings, then the kernel is contained
			in the nilradical of $R$.  Thus, if $R$ is reduced, then the map is injective. 
			
			\item\label{3.3b} For every prime ideal $P$ of $R$, either $PS = S$ or $PS$ contracts to $P$ in $R$.
			
		\end{enumerate}
	\end{proposition}
\end{samepage}
\begin{proof}
	For part \ref{3.3a}, suppose that $a \in R$ is in the kernel and not nilpotent.  By forming the quotient by a prime of $R$ not containing $a$ and the extension of this
	prime to $S$, we obtain an example where $R$ and $S$ are domains and $a$ is a nonzero element in the kernel, using Proposition~\ref{bas1}\ref{2.1f}.
	Adjoin two indeterminates $X,\, Y$ to both rings.  Then $X^2 - aY$ is prime in $R[X,Y]$:  this is true in $\cF(Y)[X]$,  where
	$\cF$ is the fraction field of $R$, and the contraction of $(X^2-aY)\cF(Y)[X]$ to $(R[Y])[X]$ is generated by $X^2 - aY$, by the
	division algorithm for monic polynomials.  The expansion of the prime $(X^2-aY)R[X,Y]$ to $S[X,Y]$ is $X^2S[X,Y]$,  a contradiction.
	
	For part \ref{3.3b}, suppose $Q' = PS \not= S$ contracts to $Q$ in $R$.  Then the map $(R/P)_Q \to (S/PS)_{Q'}$ has the \step\ by Proposition~\ref{bas2}\ref{2.2a}, and so must
	be injective by part \ref{3.3a}.  But $(Q/P)R_Q$ is in the kernel, so that $Q = P$. \end{proof}

We need one more small observation for the proof of the main theorem.

%%3.4
\begin{lemma}\label{graded-nzd}
	Let $A$ be an $\N$-graded ring with $(A_0,\fm_0)$ quasilocal. Denote by ${\fA=\fm_0  + A_{>0}}$ the unique maximal homogeneous ideal. Then every element of $A\smallsetminus \fA$ is a nonzerodivisor on $A$ and on every $\Z$-graded $A$-module.
\end{lemma}
\begin{proof}
	If $f\in A\smallsetminus \fA$, write $f=f_0+f'$, with $f_0\in A_0 \smallsetminus \fm_0$ and $f'\in A_{>0}$. There is some $h\in A_0$ with $h f_0=1$. If $fu=0$ 
	with $u \neq 0$, let $v$ be the lowest degree term of~$u$. Then $0= hfu = v + \text{higher degree terms}$, and so $v=0$, a contradiction.
\end{proof}

%%3.5
\begin{theorem}\label{thm:main} Let $R$ be a reduced ring such that every maximal ideal contains only finitely many minimal primes.
	If $R \to S$ has the \step, then $S$ is flat over $R$.\end{theorem}

\begin{proof}  First, flatness is local on the maximal ideals of $S$ and their contractions to $R$.  Hence, by Proposition~\ref{bas2}\ref{2.2a}, we may assume that
	$(R, \fm) \to (S, \fn)$ is an injective local homomorphism of quasilocal rings that has the \step, and that $(R,\fm)$ is reduced with
	finitely many minimal primes $\vect P h$.  We proceed by induction on $h$.  
	
	We need only show that $\Tor_1^R(R/I, \, S) = 0$ for all ideals $I$ of $R$, which is equivalent to the injectivity of $I \otimes S \to S$. 
	Consider $\alpha: R[It]\otimes_R S \surj S[It]$. It will suffice to prove the injectivity of the map $\alpha$: the fact that $\alpha$ is injective in degree one is exactly
	what we need.  If $h=1$,  both rings are domains containing $S$ and, therefore, $R$, and so the (non)injectivity of the map $\alpha$ is unaffected by localizing at $R\smallsetminus \{0\}$.  But then the isomorphism is clear, since $I$ becomes either the zero ideal or the unit ideal, so $\alpha$ identifies with $R\otimes_R S  \xrightarrow{\cong} S$ or $R[t] \otimes_R S \xrightarrow{\cong} S[t]$.
	
	Now assume that $k \geq 2$.  We shall prove
	$W=R \smallsetminus \bigcup_j P_j$ consists of nonzerodivisors on $R[It]\otimes_R S$. Assuming this, by Proposition~\ref{bas2}\ref{2.2a} we may localize both rings at $W$ preserving the \step\ without affecting (non)injecivity of $\alpha$ and so reduce to the case where 
	$R$ becomes a finite product of fields, and $S$ becomes a product of algebras over these fields.  The injectivity
	of $\alpha$ is local on the prime ideals of $R$.  But after localization, $I$ becomes either the zero ideal or the unit ideal, and we
	have injectivity in either case, as above.
	
	It remains to show that if $r \in W$,  then  $r$ is not a zerodivisor on $T = R[It] \otimes_R S$. We consider the $\N$-grading on $T$ induced by the usual grading on $R[It]$ (by giving $R$ degree zero and $t$ degree one) and giving $S$ degree zero. Under this grading, $T_0\cong S$ is quasilocal and thus $T$ has a unique maximal homogeneous ideal $\fN$. If $r$ is a zerodivisor on $T$, then we have a nonzero form $\phi$ in the annihilator of $r$ in~$T$.

	Let $Q_i := P_iR[t] \cap R[It]$.   Clearly, the intersection of the $Q_i$ is also zero, since it is contained in $\bigcap_i (P_i R[t]) =(\bigcap_i P_i) R[t] = 0$.  Now, ${R[It] \to S\otimes_R R[It] = T}$ has the \step\ by Proposition~\ref{bas1}\ref{2.1f},  and since $r$ is not in any $Q_iT$ and these
	ideals are prime, it follows that $\phi \in \bigcap_i Q_iT$, and it suffices to show that this intersection is zero.  To prove this, we may
	localize at the respective homogeneous maximal ideals $\fM$, $\fM'$ of $R[It]$ and $T$: by Lemma~\ref{graded-nzd}, the multiplicative systems that become inverted do not
	contain any zerodivisors, so $\bigcap_i Q_iT$ injects into its localization at $\fM'$. But we may then apply Lemma~\ref{cap} to the homomorphism
	$R[It]_\fM \to T_{\fM'}$, which has the \step\ by Proposition~\ref{bas2}\ref{2.2a},  and to the ideals $Q_iR[It]_\fM$.  We have
	$$\bigcap_i Q_iT \inc  \bigcap_i Q_iT_{\fM'} = \Big(\bigcap_i {Q_i}_\fM\Big) T_{\fM'} = \Big(\bigcap_i Q_i\Big)_\fM T_{\fM'} = 0,$$
	as required. \end{proof}

%%3.6
\begin{theorem}\label{main2} Let $h:R \to S$ be a ring homomorphism.  If $h$ has the \step, then all fibers are geometrically reduced and irreducible.
	Under the hypothesis that all fibers are geometrically reduced and irreducible, the following are equivalent:
	\begin{enumerate}[label=(\roman*)]
		\item\label{2.6a}  $R \to S$ has the \step. %% (1)
		\item\label{2.6b} $R/\fN_R \to S/\fN_R S$ has the \step, where $\fN_R$ is the nilradical of $R$.   %% (2) 
		\item\label{2.6c} For every minimal prime $P$ of $R$,  $R/P \to S/PS$ is flat. %% (3)
		\item\label{2.6d} For every finite intersection $J$ of primes of $R$,  $R/J \to S/JS$ is flat. %% (4)
	\end{enumerate} 
\end{theorem} 

\begin{proof}  The first statement was already shown in Proposition~\ref{prop:step-gri}. The equivalence of \ref{2.6a} and \ref{2.6b} is Proposition~\ref{bas2}\ref{2.2b}. If $R$ has the \step, then for any ideal $J \inc R$ that is a finite intersection of primes, $R/J$ has finitely many minimal primes, and the map $R/J \to S/JS$ has  the \step\ by Proposition~\ref{bas1}\ref{2.1f}, and hence is flat by Theorem~\ref{thm:main}. Thus, \ref{2.6a} implies \ref{2.6d}. The implication \ref{2.6d} implies \ref{2.6c} is trivial. If \ref{2.6c} holds, then $R/P\to S/PS$ has the \step\ for every minimal prime $P$, and, hence, \ref{2.6a} follows by Proposition~\ref{bas2}\ref{2.2c}.
\end{proof}

\section{A counterexample when a quasilocal ring  has  infinitely many minimal primes}\label{sec:example} \medskip

We construct $R \inj S$ local with $(R, \fm, K)$, $R,S$ quasilocal, $R$ reduced, and $S$ not flat over $R$
such that  $R \inj S$ has the \step. To do this, we will construct a non-flat module satisfying $PM=M$ for every minimal prime $P$ of $R$, 
and take $S$ to be the Nagata idealizer $R\ltimes_RM$ of $M$, which is defined to be $R \oplus M$ with multiplication $(r \oplus m)(r' \oplus m') = rr' \oplus (rm' + r'm)$,  so that $M^2 = 0$. 

We first construct an example where $R$ is $\N$-graded over a field $K$ and $S$ is $\Z$-graded.  We may then localize.  
Let $(\Sigma, \ins)$ denote a partially ordered set with the following properties: \medskip

\begin{enumerate}
	\item $\Sigma$ is nonempty.  %% (1)
	
	\item For all $\sigma \in \Sigma$, the set $\{\tau \in \Sigma: \sigma \ins \tau\}$ is finite and totally ordered.  This implies that for every element
	$\tau$, there is a unique  minimal element $\tau_0$ of $\Sigma$ with $\tau_0 \ins \tau$, and that every element $\tau \in \Sigma$ that is not minimal has a unique
	immediate predecessor, which we denote $\tau_{-}$.  We also write $\sigma \nx \tau$ to mean that $\sigma = \tau_{-}$.  %% (2)
	
	\item For every  $\sigma \in \Sigma$, there exist incomparable elements $\tau, \, \tau'$ such that $\sigma = \tau_{-} = \tau'_{-}$. %% (3)
\end{enumerate}  

The {\it height}   $\fh(\sigma)$ of $\sigma$  is the length of a maximal chain of elements descending from $\sigma$ and is one less than the cardinality
of $\{\tau \in \Sigma: \sigma \ins \tau\}$, since this set is such a chain.  We let $\Sigma_+$ denote the set of nonminimal elements of $\Sigma$.  
Let $(\Sigma, \ins)$ be a partially ordered set satisfying (1), (2), and (3) above.  Let $K$ be a field.  Let $\{X_\sigma: \sigma \in \Sigma_+\}$
be indeterminates over $K$.  Let $R = {K[X_\sigma: \sigma \in \Sigma_+]/\fI}$, where $\fI$ is the ideal generated by the products $X_\sigma X_\tau$ where $\sigma$ and
$\tau$ are incomparable.   We let $x_\sigma$ denote the image of $X_\sigma$ in $R$.   Note that the indices of variables not in a given prime ideal $Q$ must
be linearly ordered (if two were incomparable, their product is $0 \in Q$, and so at least one of them is in $Q$).  It follows at once that the minimal primes of $R$
correspond bijectively to the maximal chains $\Gamma$ in $\Sigma$, where $\Gamma$ corresponds  to
$P_\Gamma = (\{x_\gamma: \gamma \notin \Gamma\})\subset R$.  

There is a $K$-basis for $R$ consisting of products of powers of variables whose indices form a chain in $\Sigma$.

%% 4.1    
\begin{example}  If $S$ is a set with two or more elements, the set $\Sigma$ of finite sequences (including the empty sequence) of elements of $S$ with the relation that $\sigma \ins \tau$ when $\sigma$ is an initial segment of $\tau$ is an example of such a partially ordered set. In this example,  the empty sequence is the
	unique minimal element. If the empty sequence is omitted, the one element  sequences are minimal. In this example, $\sigma_-$ is the initial segment of $\sigma$  that omits the last term of $\sigma$, and the height of $\sigma$ is its length as a sequence. The minimal primes of $\fI$ are in bijection with $\N$-indexed sequences of elements of $S$. 
	
	In fact, if $S$ has two elements, the poset $\Sigma$ of this example is a subposet (up to relabeling) of any poset satisfying the three conditions above.
\end{example}

Let $\{U_\sigma: \sigma \in \Sigma\}$ be a free basis for a free $R$-module $\cM$,  and let $M$ denote the quotient of $\cM$ by the submodule spanned
by the set of elements 
\[\{\,U_{\sigma}- x_\tau U_{\tau} \,:\, \tau \in \Sigma \hbox{\ and\ } \sigma = \tau_-\,\}.\]  Let $u_\sigma$ denote the image of $U_\sigma$ in $M$.

We shall show that $R$ \and $M$, suitably localized, have the required properties.  We first explore what happens in the graded case. We introduce a ``multigrading"   as follows:  the index set will be $\Z^{\oplus\Sigma}$, the free abelian group on the elements of $\Sigma$.  The degree of $x_\sigma$ is $\sigma$.
The degree of $u_\sigma$ is $- \sum_{\tau \ins \sigma} \tau$.  Since the defining relations $X_{\sigma} X_{\tau}$ for $\sigma,\tau$ incomparable and 
$U_{\tau_-} - x_\tau U_\tau$ for all $\tau$ are multihomogeneous,  we obtain compatible gradings on $R$ and $M$. These gradings yield an $\N$-grading on 
$R$, with $R_0=K$, and a $\Z$-grading on $S$ by summimg the components of the respective multidegrees.

We first give a concrete description of $M$.

%% 4.2
\begin{proposition}\label{description-M}  With the notations introduced above, we have the following:
	\begin{enumerate}[label=(\alph*)]
		
		\item\label{4.2a} $x_\sigma u_\tau = 0$ in $M$ unless $\sigma \ins \tau$,  and $x_\sigma^2 u_\sigma = 0$. %% (a)
		
		\item\label{4.2b} If $$\tau_{k-1} \nx \tau_k \nx \tau_{k+1} \nx \,  \cdots \, \nx \tau_{n-1} \nx \tau_n$$ is a strict saturated chain of elements
		of $\Sigma$ and $a_i$ is a positive integer for all $i$ with  ${k \leq i \leq n}$, then 
		$$x_{\tau_k}^{a_k}  x_{\tau_{k+1}}^{a_{k+1}}\, \cdots \,  x_{\tau_n}^{a_n} u_{\tau_n} = 0$$ if any $a_i \geq 2$,  and is equal
		to $u_{\tau_{k-1}}$ if $a_k = a_{k+1} = \cdots = a_n = 1$.    %% (b)
		
		\item\label{4.2c} The annihilator in $R$ of $u_\tau \in M$ is the monomial ideal generated by all $x_\sigma$ such
		$\sigma \inns \tau$, and all products $x^2_{\tau_k} x_{\tau_{k+1}}\,\cdots\, x_{\tau_{n-1}}x_{\tau_n}$
		with $$\tau_k \nx \tau_{k+1} \nx\, \cdots \, \nx \tau_{n-1} \nx \tau_n$$ in $\Sigma_+.$  %% (c)
		
		\item\label{4.2d} A  $K$-basis for $M$ is given by products $\mu u_\tau$,  where $\mu$ is a monomial in the variables  $x_{\sigma}$ such that $\sigma \inns \tau$.  
		Moreover, the basis elements have mutually distinct multidegrees.  %% (d)
		
		\item\label{4.2e}  For all minimal primes $P$ of $R$,  $M = PM$. Moreover, for any $\sigma \in \Sigma_+$,  $(\Ann_R \ (x_\sigma))M = M$. %% (e)
		
		\item\label{4.2f} $M$ is not $R$-flat. %% (f)
	\end{enumerate}
\end{proposition}

\begin{proof}
	
	For part \ref{4.2a}, note that if $\sigma\not\ins \tau$, then either $\sigma,\tau$ are incomparable or ${\tau\inns \sigma}$. If $\sigma,\tau$ are incomparable, and 
	$\gamma$ is an immediate successor 
	of $\tau$, then $\sigma,\gamma$ are incomparable. Then $x_{\sigma}u_{\tau}=x_\sigma (x_{\gamma} u_{\gamma})=0$. Similarly, if $\tau\inns \sigma$, there is an immediate 
	successor of $\tau$ that is not comparable with $\sigma$, and the relation follows in the same way. This justifies the first statement. From this, we have 
	$x_{\sigma}^2 u_{\sigma}=x_{\sigma} u_{\sigma_-}=0$.
	
	For part \ref{4.2b}, we use induction on $n-k$.  If $k=n$, the result follows at once from the final statement of part \ref{4.2a}, and this is also true if
	$a_n \geq 2$.  If $a_n = 1$, we may replace final part of the product, consisting of $x_{\tau_n}u_{\tau_n}$, by $\tau_{n-1}$, and then the
	result follows from the induction hypothesis.  
	
	We prove \ref{4.2c} and \ref{4.2d} simultaneously. Note that the monomials specified in \ref{4.2c}  kill $u_\tau$ by part \ref{4.2b}. It is easy to see from part \ref{4.2a} that $M$ is spanned as a $K$-vector space by the terms $\mu u_\tau$ such that all variables occurring in $\mu$ having subscripts strictly less than $\tau$; we will call the set of such expressions $\cB$. For a fixed $\omega\in \Sigma$, we will write $\cB(\omega)$ for the elements $\mu u_\tau \in \cB$ with $\tau \ins \omega$. It is easy to recover an element of $\cB$ in $M$ from its multidegree: $\tau$ will be the unique largest element with a negative coefficient in the multidegree, and the exponent on each variable $x_\sigma$ occurring in $\mu$ will be one greater than the coefficient of $\sigma$ in the multidegree. Thus, each multigraded component of $M$ is at most one-dimensional as a $K$-vector space. However, it remains to see that the elements in $\cB$ are all nonzero in $M$.
	
	Let $V$ be the  ``formal" $K$-vector space spanned by the terms $\mu u_\tau$ in $\cB$.  We shall give an $R$-module structure to $V$
	by specifying a $K$-linear endomorphism $\theta_\tau$ of $V$ for every $\tau \in \Sigma_+$ such that the following two conditions hold:
	\begin{enumerate}
		\item[($\dagger$)]\label{cond-1-theta} For all $\sigma, \tau \in \Sigma_+$,  $\theta_\sigma \circ \theta_{\tau} = \theta_\tau \circ \theta_\sigma$,  %% (1)
		\item[($\ddagger$)]\label{cond-2-theta} For all $\sigma, \tau \in \Sigma_+$, if $\sigma$ and $\tau$ are incomparable, $\theta_\sigma \circ \theta_{\tau} = \theta_\tau \circ \theta_\sigma = 0$. %% (2)
	\end{enumerate}  
	
	These conditions give $V$ the structure of an $R$-module.  We then verify $V \cong M$ as $R$-modules in such a way that the formal basis element $\mu u_\tau$
	corresponds to the term $\mu u_\tau$ in $V$.  
	
	Let $\fh(\tau)  = n$, and let $\tau_0 \nx \tau_1 \nx \cdots \nx \tau_{n-1} \nx \tau_n = \tau$ be the chain of elements that are $\ins \tau$.  
	For each monomial $\mu$ in the $x_\sigma$ for $\sigma \inns \tau$, let $k$ be one plus the largest index $j$ such that $x_{\tau_{j}}$ does not occur 
	with positive degree in $\mu$, which we take to be $0$ if all the variables occur; note that $k$ depends on $\mu$, but we omit it from the notation for ease of reading.  Then $\mu$  can be written uniquely in the form  $\nu_\mu \gamma_\mu$  where
	\begin{align*} \gamma_\mu \quad &\text{is a monomial in the variables indexed by} \quad \tau_{k}, \,\tau_{k+1},  \dots, \, \tau_{n-1} \quad \\
	&\text{and each such variable divides}  \ \ \gamma_\mu,
	\intertext{and}
	\nu_\mu \quad  &\text{involves only variables with subscripts} \quad \tau_1, \, \ldots, \tau_{k-2}.
	\end{align*}
	Note that if $x_{\tau_{n-1}}$ does not occur with a positive exponent, then $\nu_\mu = \mu$ and 
	$\gamma_\nu = 1$,  while if for $1 \leq j \leq n$ each $x_{\tau_j}$  occurs with positive exponent then $\nu_\mu = 1$ and $\gamma_\mu = \mu$.  
	For $\sigma\in \Sigma_+$, let $\theta_{\sigma}$ be the $K$-linear endomorphism of $V$ specified on the basis $\cB$ by 
	
	\[ \theta_{\sigma} (\mu u_{\tau}) = 
	\begin{cases} 
	x_{\sigma}\mu u_{\tau}&  \mbox{if } \sigma \inns \tau,\\
	\nu_\mu u_{\tau_{k-1}} & \mbox{if } \sigma=\tau \ \text{and all exponents in }\gamma_{\mu} \text{ are } 1, \\ 
	0 & \mbox{if } \sigma=\tau \ \text{and some exponent in }\gamma_{\mu} \text{ is }\geq 2, \\ 
	0  &\mbox{if } \sigma \not\ins \tau.   \\ 
	\end{cases}\] 
	
	We need to check that ($\dagger$) and ($\ddagger$) hold. Fix $\tau = \tau_n$ of height $n$. 
	Note that $\theta_{\sigma}$ stabilizes the $K$-span of $\cB(\tau)$,  which we denote $W(\tau)$,
	and that all $x_\sigma$ kill these elements unless $\sigma \ins \tau$.  It therefore suffices to consider only the variables $x_{\tau_i}$ for  $1 \leq i \leq n$,
	and the $u_{\tau_i}$ for $0 \leq i \leq n$.  To simplify notation, we shall write $i$ instead of $\tau_i$, so that we have $0 \nx 1 \nx \, \cdots \, \nx n$.
	To verify ($\dagger$) and ($\ddagger$) consider the $R$-module $R/\fA$  where $\fA$ is the monomial ideal generated by all variables whose subscripts are not $\ins n$
	and all monomials of the form $x_k^2x_{k+1} \cdots x_n$.  The quotient has a $K$-basis consisting of all monomials of the form $\nu \gamma_k$  
	for $0 \leq k \leq n$ where $\gamma_k = \prod_{j=k+1}^n x_j$ and $\nu$ is a monomial in $\vect x {k-1}$. 
	
	There is a vector space isomorphism between 
	$R/\fA$ and $W(\tau)$ that maps $\nu \gamma_k$ to $\nu \tau_k$.  It is straightforward to verify that
	for $1 \leq j \leq n$,  $x_j \nu\gamma_k$ corresponds to $\theta_j(\nu \tau_k)$.  It follows that ($\dagger$) and ($\ddagger$) hold for the $\theta_j$ acting on
	$W(\tau)$, and thus ($\dagger$) and ($\ddagger$) hold on $V=\varinjlim_{\tau} W(\tau)$. This gives $V$ the structure of an $R$-module. 
	
	The $R$-linear map from the free module $\cM$
	to $V$ such that $U_\tau \mapsto  1\cdot u_\tau$ kills the relations defining $M$, and so induces an $R$-linear surjection $M \surj V$.  Since
	the elements in $M$ corresponding to $\cB$ span $M$,  these elements are also a $K$-basis for $M$:
	if they were linearly dependent, their images in $V$ would be as well.  It follows that the map $M \surj V$ is an $R$-isomorphism.  This proves
	part \ref{4.2d}, while part \ref{4.2c} now follows from the fact that $Ru_\tau \cong W(\tau) \cong R/\fA$.  
	
	Part \ref{4.2e} is clear, since every $\sigma \in \Sigma$ has at least two incomparable immediate successors $\tau$, $\omega$,  and given any minimal prime $P =P_{\Gamma}$, at
	least one of $x_\tau$ or $x_{\omega}$ is in $P_\Gamma$, say $x_\tau$, and then $u_\sigma = x_{\tau} u_{\tau} \in PM$.  
	
	It remains only to prove \ref{4.2f}. Fix $\sigma \in \Sigma_+$, and consider the exact sequence  
	$$\CD 0 @>>> \Ann_R (x_\sigma) @>>> R @>{x_{\sigma}}>> R \endCD.$$  
	If $M$ were $R$-flat then tensoring with
	$M$ would yield that $$\Ann_M (x_\sigma) = (\Ann_R (x_\sigma))M.$$ We have
	\[ \Ann_R (x_\sigma)  = \big(\{x_\tau: \tau \in \Sigma_+ \ \text{ and } \ \sigma, \, \tau \ \text{ are incomparable }\}\big).\] 
	Suppose $\sigma \nx \sigma_1$.  Then $x_\sigma u_{\sigma_1}$ is a nonzero element of $M$ in $\Ann_M (x_{\sigma})$.  Hence, if $M$ were $R$-flat we would have that $x_\sigma u_{\sigma_1}= \sum_j x_{\tau_j} \mu_j u_{\omega_j}$ with $\tau_j$ incomparable to $\sigma$ and $\mu_j$ a monomial for each $j$. Using the $\Z^{\oplus \Sigma}$ grading, and the fact that each piece is a one-dimensional $K$-vector space, there must be an equality of the form $x_\sigma u_{\sigma_1}=x_{\tau} \mu u_{\omega}$,  with $\sigma, \tau$ incomparable, and $\tau \ins \omega$, for otherwise the right-hand side would be zero. Using the module structure as computed above, we can write $x_{\tau} \mu u_{\omega}=\mu_1 u_{\omega_0}$ with $\omega_0 \ins \omega$. Thus we have $x_\sigma u_{\sigma_1}=\mu_1 u_{\omega_0}$ in $M$. Since the elements of $\cB$ form a basis, we must have $\sigma_1=\omega_0$. But then, we have $\sigma \ins \sigma_1 \ins \omega$, and $\tau \ins \omega$, so $\sigma$ and $\tau$ are comparable, a contradiction. \end{proof}

%% 4.3
\begin{theorem} Let $R$ and $M$ be as above. Let $\fm$ be the homogeneous maximal ideal of $R$. Then $R_{\fm}\to (R\ltimes M)_{\fm (R\ltimes M)}$ is a local homomorphism of quasilocal rings, with source $R_{\fm}$ reduced, that satisfies the \step, but is not flat.
\end{theorem}
\begin{proof}
	Since $R$ is reduced, $R_{\fm}$ is as well. Observe that $R\to R\ltimes M$ has the \step\ by  Proposition~\ref{bas2}\ref{2.2c}, since for any minimal prime $P$, we have
	${(R\ltimes M)/P(R\ltimes M)\cong R/P}$. 
	Then, by Proposition~\ref{bas2}\ref{2.2a}, the map $R_{\fm}\to (R\ltimes M)_{\fm (R\ltimes M)}$ has the \step\ as well.  It remains only to show that
	$M_\fm$ is not flat over $R_\fm$.  We can argue in the same way:  if $M_\fm$ were flat, we would have for every $\sigma \in \Sigma_+$
	that  $\Ann_{M_\fm} (x_\sigma)/(\Ann_R (x_\sigma))M_{\fm} = 0$,  and since localization commutes both with taking the annihilator of an element
	and with forming quotients, this would imply $\bigl(\Ann_M (x_\sigma)/(\Ann_R (x_\sigma))M\bigr)_\fm = 0$.  To see that this does not happen we
	may apply Lemma~\ref{graded-nzd}, using the $\N$-grading on $R$ and the $\Z$-grading on $M$ introduced before the statement of Proposition~\ref{description-M}.
\end{proof}   

\section{Intersection Flatness} %%5

Recall (cf.~\cite[p.~41]{HH}) where the terms ``intersection-flat" and ``$\cap$-flat are used)
that an $R$-module $S$ is {\it intersection flat} if $S$ is $R$-flat and for any finitely generated $R$-module $M$, \medskip
\begin{enumerate}
\item[(\#)]        for every 
	family of submodules  $\{M_\la\}_{\la \in \Lambda}$ of $M$ we have 
	$S\otimes_R (\bigcap_\la M_\la)$ $= \bigcap_\la (S\otimes_R M_\la)$ when both are identified with their images in $S\otimes_R M$.
\end{enumerate}
\smallskip

Note that, quite generally, there is an obvious injective map from the first module to the second. Here, $S$ will usually be an $R$-algebra 
in which case  we also say the homomorphism $R \to S$ is {\it \infl}.   A flat homomorphism satisfies the property $(\#)$ for any module $M$ whenever $\Lambda$ is finite.

In particular, if $S$ is \infl, then for every family of ideals ${\{ I_{\la} \, : \, \la \in \Lambda\}}$ of $R$ the equality $(\bigcap_{\la} I_\lambda)S = \bigcap_{\la} (I_\la S)$ holds. We shall say that $S$ (or $R \to S$ in the algebra case) is \emph{weakly \infl\ for ideals} if this condition holds, i.e., if ($\#$) holds when $M = R$, and \emph{\infl\ for ideals} if additionally $R\to S$ is flat. We caution the reader that the definition of ``intersection flat'' given in \cite{AH} is the notion that we call ``\infl\ for ideals" here. Note that when $M = R$, we may identify $I\otimes_R S$ with $IS$.

The notion of weak intersection flatness for ideals has been previously studied in the context of the theory of content. If $S$ is a weakly \infl\ $R$-module, then, for any $s\in S$, there is a unique smallest ideal $I$ of $R$ such that $s\in IS$, called the \emph{content} of $s$; in fact, this characterizes the property of weak intersection flatness for ideals by \cite[1.2]{OR}. Note that if $S$ is a polynomial ring over $R$, then this notion of content coincides with the classical notion of content as the ideal generated by the coefficients of a polynomial. A module $S$ that is weakly intersection flat for ideals is called a \emph{content module} in \cite{OR,Ru}, and called an \emph{Ohm-Rush module} in \cite{ES1,ES2}.

The relationships between the notions in the title of this paper have been explored in the works cited in the previous paragraph, and some statements related to parts of the lemmas below appear in these sources. We note the following result of Rush \cite[Theorem~3.2]{Ru}, which should be compared with Theorem~\ref{thm:main} and Proposition~\ref{equiv-Iflat} below.

\begin{theorem}[Rush]
	Let $\varphi: R\to S$ be an injective ring homomorphism that is weakly intersection flat for ideals. Suppose that $R$ is reduced and that $\varphi$ satisfies the \pep. Then $\varphi$ is flat.
	\end{theorem}

For the most part, it is intersection flatness for ideals that we use for results on the \step\ (and, as mentioned above, 
is the definition used in \cite{AH}).
As is shown in Proposition~\ref{int-basic}\ref{i-flat-dir-sum} below, if  $\phi:R\to S$ is module-free (i.e., $S$ is a free $R$-module) then $\phi$ is intersection flat \cite[p.~41]{HH}.
Proposition~\ref{int-basic} below collects some facts about intersection flatness. Note that intersection flatness for the Frobenius endomorphism
mapping a regular ring to itself is studied in \cite[5.3]{Ka} and in \cite[\S9]{Sh}. 

%% 5.1
\begin{example}\label{notmod} Let $K$ be a field and let $R = K[x]_{\fm}$ where $\fm = xK[x]$.  Let $f$ be a formal power series in 
	$x$ that is not in the fraction field $K(x)$ of $K[x]$.
	Then $K\llbracket x\rrbracket$ is \infl\ for ideals over $R$, since the intersection of any infinite family of ideals is $(0)$ in both rings.  However, $K\llbracket x\rrbracket$ is {\it not}
	intersection flat over~$R$.  To see this, let $f_n$ denote unique the polynomial of degree at most $n$ that agrees with $f$ modulo $x^{n+1}K\llbracket x \rrbracket$.
	Let $M = R^{\oplus 2}$,  and  let $M_n= R(1, f_n) + R(0,x^{n+1})$.  When we tensor with $\wh{R} = K\llbracket x \rrbracket$, then $\wh{R} \otimes_R M_n$ may
	be identified with the submodule of $\wh{R}^{\oplus 2}$ spanned by $(1,f)$ and $(0,x^{n+1})$.  The intersection of these is $\wh{R}(1,f)$.  But the intersection
	of the $M_n$ in $R^2$ is 0: if $(g,h)\in M_n$, then $h - fg\in \fm^n$, so if $(g,h)$ were a nonzero element of the intersection, we would have  that $f = g/h$ would be rational over $K[x]$. \end{example}

% 5.2
\begin{example} $K\llbracket x,\,y\rrbracket$ is not even \infl\ for ideals over $R = K[x,\,y]_{\fm}$,  where ${\fm}$ is the maximal ideal $(x,\,y)$ in the polynomial ring $K[x,y]$.
	Let $f \in K\llbracket x \rrbracket$ be transcendental over $K[x]$ and let $f_n$ be as in Example~\ref{notmod}.  Then the ideals $I_n = (y - f_n, x^{n+1})R$ when extended
	to  $\wh{R} = K\llbracket x,y\rrbracket$ agree with $(y-f, \, x^{n+1})\wh{R}$, and so their intersection is $(y-f)\wh{R}$.  But their intersection in $R$ is 0, for
	if $0 \not= G(x,\,y) \in R$ were in the intersection we may  clear the denominator and assume that $G \in K[x, y]$,  and then
	we would have $G(x, f) = 0$, contradicting the transcendence of $f$.  \end{example}

%% 5.3
\begin{discussion}\label{altint} Note that a family of $R$-submodules of $M/N$  has the form $M_\lambda/N$ where the $M_\la$ are submodules of $M$ containing $N$,
	and we have
	$$\bigcap_\la(M_\la /N) = \big(\bigcap_\la M_\la\big)/N \quad \hbox{\ and so}$$
	$$\bigcap_\la \bigl(S \otimes_R (M_\la/N)\bigr) = \bigcap_\la \bigl((S \otimes_R M_\la)/(S \otimes_R N)\bigr) = \biggl(\bigcap_\la (S \otimes_R M_\la)\biggr)/(S \otimes_R N).$$
	Under the assumption that  $R\to S$ is intersection flat, this identifies with
	\[ \Biggl(S \otimes_R\biggl( \bigcap_\la M_\la\biggr)\Biggl) \Bigg/ (S \otimes_R N) =   S \otimes_R \Biggl(\biggl(\bigcap_\lambda M_\la\biggr)\bigg/N\Biggr) = S \otimes_R\bigl(\bigcap_\la (M_\la/N)\bigr),\]
	where the equalities are canonical identifications of submodules of $S \otimes_R M$ or of $S \otimes_R (M/N)$. 
\end{discussion}

The following result clarifies when properties related to intersection flatness over $R$ imply flatness.  One consequence is that in the definition
of intersection flat, it is not necessary to assume flatness separately:  the property $(\#)$, even for finite families of submodules of $R^2$, implies
it. 

%% 5.4
\begin{proposition}\label{equiv-flat}  Let $R$ be a ring and $S$ an $R$-module.  The following are equivalent.
	\begin{enumerate}[label=(\roman*)]
		\item\label{fl-1} $S$ is $R$-flat.
		\item\label{fl-2} For every $R$-module $M$, the property $(\#)$ holds for every finite family of submodules of $M$.
		\item\label{fl-3} The property $(\#)$ holds for every finite family of submodules of  $M = R^2$ (and, hence, for every
		finite family of ideals in $M = R$).
		\item\label{fl-4} For every element $f \in R$,  $\Ann_S (f) = (\Ann_R (f))S$, and the property $(\#)$ holds for every finite family
		of ideals of $M = R$.
       \end{enumerate}
\end{proposition}
\begin{proof}  The implication \ref{fl-1} $\Rightarrow$ \ref{fl-2} is well-known, and \ref{fl-2} $\Rightarrow$ \ref{fl-3}
is clear.  To show that \ref{fl-3} implies \ref{fl-4}, consider the submodules spanned by $(1,0)$ and $(1,f)$ in $R^2$.  Their intersection is 
	$\Ann_R (f) \times 0$. If we expand to $S$, the intersection is $\Ann_S (f) \times 0$, while the expansion of the intersection is 
	$(\Ann_R (f))S \times 0$,  from which the result follows. 
	
	It remains to show \ref{fl-4} $\Rightarrow$  \ref{fl-1}; i.e., that  \ref{fl-4} implies that $S$ is flat.
	It suffices to show that for every finitely generated ideal $I = (\vect f n)R$ of $R$,  we have that $\Tor_1^R(R/I,\,S) = 0$.
	%(By a direct limit argument the same then holds for all $R/I$ whether or not $I$ is finitely generated, and so for every finitely generated module, since these have finite filtrations with cyclic factors, and thence for every module, by another direct limit argument.)  
	This follows
	if every relation $\sum_{i=1}^n f_i s_i = 0$ is an $S$-linear combination of relations on $\vect f n$ over $R$.  We prove this by induction
	on $n$.  If $n = 1$, this follows from the fact that $\Ann_S (f_1) = (\Ann_R (f_1))S$.  
	Now suppose we know the result for $\vect f {n-1}$, $n \geq 2$. 
	The relation $\sum_{i=1}^n{f_is_i = 0}$ implies that 
	$$(\star) \qquad -f_ns_n = \sum_{i=1}^{n-1} f_is_i$$ 
	and this is element is in $(\vect f {n-1})S \cap f_nS = \bigl((\vect f {n-1})R \cap f_nR\bigr)S$,  and so the equation in ($\star$) arises as an $S$-linear combination, say with coefficients $t_j \in S$, $1 \leq j \leq h$, of equations
	$$(\star) \qquad -f_nr_{nj} = \sum_{i=1}^{n-1} f_ir_{ij}$$ where the $r_{ij} \in R$.  These may be rewritten as relations on $\vect f n$
	with coefficients $r_{ij}$ in $R$.  After multiplying these relations by the $t_j$,  adding, and subtracting the sum from the original
	relation, we get a new relation on $\vect fn$ with coefficients in $S$,  say $\sum_{i=1}^{n-1} f_iu_i +f_nu_n = 0$, in which
	both $\sum_{i=1}^{n-1} f_iu_i = 0$ and $f_nu_n = 0$.  Using the induction hypothesis for the former and the case $n=1$
	for the latter, we see that this relation is an $S$-linear combination of relations over $R$.
	\end{proof}

%% 5.5
\begin{proposition}\label{equiv-Iflat}
	Let $R$ be a ring, and $S$ be an $R$-module. The following are equivalent.
	\begin{enumerate}[label=(\roman*)]
		\item\label{ifl-1} $S$ is an intersection flat $R$-module.
		\item\label{ifl-2} For every finitely generated $R$-module $M$, the property ($\#$) holds.
		\item\label{ifl-3} For every finitely generated free $R$-module $M$, the property ($\#$) holds.
		\item\label{ifl-4} For every finitely generated $R$-module $M$ and every family of submodules $\{ M_{\lambda} \, : \, \lambda\in \Lambda\}$ such that $\bigcap_{\la} M_{\lambda}=0$, we have $\bigcap_\la (S \otimes_R M_\la) = 0$.
	\end{enumerate}
\end{proposition}
\begin{proof}
	The implications \ref{ifl-1} $\Rightarrow$ \ref{ifl-2} $\Rightarrow$ \ref{ifl-3} are clear. The equivalence between \ref{ifl-2} and \ref{ifl-4} is immediate from Discussion~\ref{altint} by replacing $M$ by $M/N$ and every $M_\lambda$ by $M_\lambda/N$. Similarly, for \ref{ifl-3} $\Rightarrow$ \ref{ifl-2}, given a module $M$ and a family of submodules, we map a free module $G \surj M$ and work with the family consisting of the inverse images of the $M_\lambda$ in $G$.  We may then apply Discussion~\ref{altint}. 
	It remains only to show \ref{fl-4} $\Rightarrow$  \ref{fl-1}; i.e., that  \ref{fl-4} implies that $S$ is flat. But this is immediate from the
	implication \ref{fl-2} $\Rightarrow$ \ref{fl-1} in Proposition~\ref{equiv-flat},   
		\end{proof}

%% 5.6
\begin{proposition}\label{int-basic} Let $R$ be a ring.
	\begin{enumerate}[label=(\alph*)]
		\item\label{ib-a} If $R \to S$ is intersection flat (respectively, \infl\ for ideals) and $T$ is an $S$-module that is \infl\ (respectively, \infl\ for ideals) over $S$, then $T$ is \infl\ (respectively, \infl\ for ideals) over~$R$.  %% a
		\item\label{i-flat-dir-sum}\label{ib-b} Direct sums and direct summands of modules that are intersection flat (\rifi) are intersection flat (\rifi).%b
		\item\label{ib-c} If $R$ is Noetherian, arbitrary products of modules that are intersection flat  (\rifi) are intersection flat (\rifi). %% c
		\item\label{ib-d} The formal power series ring in finitely many variables over a Noetherian ring is intersection flat. %% d
		\item\label{ib-e} Let $R$ be a  complete local Noetherian ring and $S$ an $R$-flat algebra.  If $S/IS$ is $\fm$-adically separated  for every         ideal $I$ of $R$, then $S$ is \infl\ for ideals.  If $S \otimes_R M$ is  $\fm$-adically separated for all finitely generated $R$-modules $M$  then $R \to S$ is \infl.  In particular, if $S$ is Noetherian and $\fm S$ is contained in the Jacobson radical
		of $S$,  then $R \to S$ is \infl.  %% e
		\item\label{ib-f} If $S$ is intersection flat (\rifi) over $R$ and  $W$ is a multiplicative system in $S$ such that
		no element of $W$ is a zerodivisor on $S/IS$ for any ideal $I$ of $R$,  then $W^{-1}S$ is intersection
		flat (\rifi) over $R$.  In particular,  if we take $S$ to be a polynomial ring in an arbitrary set of variables over $R$
		and $W$ to consist of a set of polynomials, each of which has a set of coefficients that generates the unit ideal of $R$, then $W^{-1}S$ is intersection flat over $R$.   %% f
	\end{enumerate}
\end{proposition}
\begin{proof}  
	The arguments dealing with the case with the property ``intersection flat for ideals" are identical with
	those for the module case, and are not given separately, except for a remark in the proof of part \ref{ib-e}.
	
	Part \ref{ib-a} is immediate from the definition, and part \ref{ib-b} is a completely straightforward consequence of the fact that for a
	direct sum $S = \bigoplus_\mu S_\mu$, we have
	$$(\bigcap M_\lambda) \otimes_R (\bigoplus_\mu  S_\mu) \cong \bigoplus_\mu(\bigcap_\la M_\la) \otimes_R S_\mu,$$
	which will be equal to 
	$$\bigcap_\la \bigl(M_\lambda \otimes_R (\bigoplus_\mu S_\mu)\bigr) \cong  
	\bigcap_\la \bigl(\bigoplus_\mu (M_\lambda \otimes_R S_\mu)\bigr) \cong 
	\bigoplus_\mu \bigl (\bigcap_\lambda\bigl(M_\lambda \otimes_R S_\mu)\bigr).$$
	
	To prove \ref{ib-c}, note that over a Noetherian ring, an arbitrary product of flat modules is flat, by Chase's theorem \cite[Theorem 2.1]{Ch}.  It suffices to prove the property $(\#)$ for an arbitrary finitely generated $R$-module $M$ (or when $M = R$ in the ideal case).  
	This follows from the observation that for a {finitely presented} $R$-module $N = M_\lambda$, the map
	$N\otimes_R(\prod_\mu S_\mu) \to \prod_\mu (N \otimes_R S_\mu)$ is an isomorphism.  (Using the finite presentation of $N$, we 
	reduce to the case where $N = \bigoplus_{i=1}^k Re_i$  is free.  If an element of $\prod_\mu  (N \otimes_R S_\mu)$ has 
	$\mu$-coordinate $\sum_{i=1}^k s_{i,\mu}e_i$, we let it correspond to $\sum_{i=1}^k \sigma_i e_i$ in $N \otimes_R \prod_\mu S_i$,
	where $\sigma_i$ has $\mu$-coordinate $s_{i,\mu}$.)
	
	Part \ref{ib-d} is immediate from part \ref{ib-c}, since these power series rings, as $R$-modules, are products of copies of $R$.
	
	To prove part \ref{ib-e}, we note that since flatness implies that extension commutes with finite intersection, we may assume the
	family $M_\la$ is closed under finite intersection.  Let $N$ denote the intersection of the family.  Chevalley's lemma implies
	that for every $C \in \N$,  there exists $\la_C \in \Lambda$ such that $M_{\la_C} \inc N + \fm^C M$.  It follows that 
	\[S \otimes_R M_{\la_C} \inc  S \otimes_R N +S\otimes_R \fm^C M \subseteq S\otimes_R N + \fm^C (S\otimes_R M).\]  Since $S\otimes_R (M/N)\cong (S\otimes_R M) / (S \otimes_R N)$ is $\fm$-adically separated by hypothesis, the intersection of the modules on the right-hand side as $C$ varies is $S\otimes_R N$, so the intersection of $\{ S \otimes_R M_{\la} \}$ is $S\otimes_R N$.
	In the ideal case, we only need that for every ideal $I$ of $R$,  $S/IS$ is $\fm$-adically separated. For the last statement of \ref{5.5e}, we simply recall that every finitely generated $S$-module is separated with respect to the Jacobson radical of $S$.
	
	For part \ref{ib-f}, note that no element of $W$ is a zerodivisor on any module of the form $S \otimes_R N$, where $N$ is an $R$-module:  it suffices to consider finitely generated $R$-modules $N$, and these have finite filtrations in which the factors have the form  $R/I$. 
	
	Now consider an element of $\bigcap_\la (W^{-1}S \otimes_R \, M_\la)$ in $W^{-1}S \otimes_R M$:  after multiplying by a unit from the 
	image of $W$, we may
	assume this element is the image of an element $u \in S \otimes_R M$.  Then $u/1$ is in every $W^{-1}(S \otimes_R M_\la)$,  and since
	$W$ consists of nonzerodivisors on $S \otimes_R(M/M_\la)$, we have that $u$ is in every $S \otimes_R M_\la$. By the
	hypothesis on $S$,  this implies that $u \in \bigl(S \otimes_R (\bigcap_\lambda M_\la)\bigr)$, and so $u/1 \in W^{-1}S \otimes_R (\bigcap_\la M_\la)$,
	as required.
\end{proof}

Note that part \ref{ib-e} of Proposition~\ref{int-basic} considerably strengthens the result of \cite[Proposition~5.3]{Ka}. 

%% 5.7
\begin{proposition}\label{prop:int-flat} If $R \to S$ is intersection flat for ideals and $P$ is the intersection of a family of primes $\{Q_\lambda \ : \ \lambda\in \Lambda\}$ in  $R$ such that $S/(Q_\lambda S)$ is zero or a domain for each $\lambda$, then $S/PS$ a domain.
\end{proposition}
\begin{proof} 
	If $fg \in PS$, then for each $\lambda$, we have either $f  \in Q_\lambda S$ or $g \in Q_\lambda S$. Set $\Lambda=\Lambda_f \cup \Lambda_g$, where $f  \in Q_\lambda S$ for $\lambda\in \Lambda_f$ and $g  \in Q_\lambda S$ for $\lambda\in \Lambda_g$. Then, if $A=\bigcap_{\lambda\in \Lambda_f} Q_{\lambda}$ and $B=\bigcap_{\lambda \in \Lambda_g} Q_{\lambda}$ we have radical ideals $A, B$ for which
	$A \cap B = P$.  We have  $f \in AS$ and $g \in BS$.  But one of $A,B$ must be $P$, so $f \in PS$ or $g \in PS$. 
\end{proof}

We recall that a \emph{Hilbert ring} is a ring in which every prime ideal is an intersection of maximal ideals.

%% 5.8
\begin{proposition}\label{prop:Hilbert-pep}
	If  $R$  is a Hilbert ring,  $R \to S$  is intersection flat for ideals, and for every
	maximal ideal  $\fm$ of $R$,  $S/\fm S$  is zero  or a domain, then 
	$R$ has the \pep.
\end{proposition}
\begin{proof}
	This is immediate from Proposition~\ref{prop:int-flat}, since every prime is an intersection of maximal ideals by definition.
\end{proof}

Using this, we can give some versions of Proposition~\ref{prop:step-gri} that refer only to closed fibers.

%% 5.9
\begin{proposition}\label{prop:Hilbert-step}
	Let $R$ be a Hilbert ring, $R\to S$ be module-free, and suppose that for every maximal ideal $\fm$ of $R$, the fiber $\kappa_{\fm}\otimes_R S$ is a geometrically reduced and irreducible $\kappa_{\fm}$-algebra. Then $R\to S$ has the \step.  
\end{proposition}
\begin{proof}
	Since $R\to S$ is module-free, by base change
	\[R[\vect X n]\to S \otimes_R R[\vect X n] \cong S[\vect X n]\]
	is  module-free as well, and hence intersection flat. By the Hilbert hypothesis, any maximal ideal $\fM$ contracts to a maximal ideal $\fm$ of $R$. We then have that 
	\[ \frac{S[\vect X n]}{\fM S[\vect X n]} \cong \kappa_{\fM} \otimes_R S \cong (\kappa_{\fm} \otimes_R S) \otimes_{\kappa_{\fm}} \kappa_{\fM} \]
	and by the hypothesis on the fibers, this quotient is either a domain or zero. By Proposition~\ref{prop:Hilbert-pep}, $R[\vect X n]\to S[\vect X n]$ then has the \pep, as required.
\end{proof}

%% 5.10
\begin{corollary}\label{affstep}
	Let $R$ be a finitely generated algebra over an algebraically closed field, and let $S$ be a nonzero module-free $R$-algebra. Suppose that for every maximal ideal $\fm$ of $R$, we have $S/\fm S$ is a domain. Then $R\to S$ has the \step. 
\end{corollary}
\begin{proof}
	In this case $R$ is a Hilbert ring, and every residue field is algebraically closed. Thus, if the fiber $\kappa_{\fm} \otimes_R S$ is a domain or zero, then it is a \gri\ $\kappa_{\fm}$-algebra, by, e.g., \cite[Propositions~4.5.1 and~4.6.1]{EGA-IV2}. Proposition~\ref{prop:Hilbert-step} then applies.
\end{proof}

%% 5.11
\begin{theorem}\label{gr-reg} Let  $K$ be an algebraically closed field and let $S$ be an $\N$-graded $K$-algebra that is a domain.  
	Let $\vect Fn$ be a regular sequence of forms in $S$ generating a prime ideal $Q$ of $S$. 
	Then the $K$-algebra map of the polynomial ring $R = K[\vect Xn] \to S$  such that $X_i \mapsto F_i$, $1 \leq i \leq n$, has the \step. \end{theorem}

\begin{proof}   The hypothesis is stable under adjoining finitely many indeterminates to both rings:  we may use these to enlarge the sequence
	$\vect F n$.  Thus, it suffices to prove the \pep\ under the given hypotheses.  The hypothesis implies that $S$ is free over $R$. Indeed, fix
	a homogeneous basis for $S/QS$ over $K$.  These elements will span $S$ over $R$ by the graded version of Nakayama's lemma.
	They have no relations by induction on $n$:  given a nonzero relation we may factor out the highest power of $F_1$ occurring in all coefficients,
	since $F_1$ is a nonzerodivisor in $S$, and then we obtain a nonzero relation on the images of these generators working over
	$S/F_1S$ and $R/X_1 R$. Since every maximal ideal of $R$ has the form $(X_1 - c_1, \, \ldots, X_n-c_n)$  for $\vect c n \in K$, we know that the expansion of any maximal ideal of $R$ to $S$ has the form $(F_1 - c_1, \, \ldots, F_n-c_n)$. The result now follows from Corollary~\ref{affstep} and \cite[Proposition~2.8]{AH}.  
\end{proof}

\section*{Acknowledgments}

We thank Neil Epstein for bringing to our attention the connection between intersection flatness and the theory of content in the sense of Ohm and Rush. We also thank Gabriel Picavet for directing us to his earlier work on the \step\ and flatness for domains.

%%%%%%%%%%%%%%%%%%%%%%%%%%%%%%%%%%%%%%%%%%%%%%%%%%%%%%%%%%%%%%%%%%%%%%%%%%%%%%%%%%%%%%%%%%%%%%%%%%%%%%%%%%%%%%%%%%%%%%%%%%%%%%%%%%%%%%%%%%%%%%%%%%%%%%%%%%%%%%%%%%%%%%%%%%%%%%%%%%%%%%%%%%%%%%%%%%%%%%%%%%%%%%%%%%%%%%%%%%%%%%%%%%%%%%%%%%%%%%%%%%%%%%%%%%%%%%%%%%%%%%%%%%%%%%%%%%%%%%%%%%%%%%%%%%%

\bibliographystyle{amsalpha}

\begin{thebibliography}{A}

\bibitem{AH} T.~Ananyan and M.~Hochster,  \emph{Small subalgebras of polynomial rings and Stillman's conjecture}, 
J.~Amer. Math.~Soc., {\bf 33} (2019) 291--309.

\bibitem{Ch} S.~U.~Chase, \emph{Direct products of modules},
Trans.~Amer.~Math.~Soc.~{\bf 97} (1960), pp.~457--473.

\bibitem{ES1} N.~Epstein and J.~Shapiro, \emph{The Ohm-Rush content function}, J.~Algebra~Appl.~\textbf{15} (2016), 1650009, 14~pp.

\bibitem{ES2} N.~Epstein and J.~Shapiro, \emph{The Ohm-Rush content function II. Noetherian rings, valuation domains, and base change.} J.~Algebra~Appl.~\textbf{18} (2019), 1950100, 23~pp.

\bibitem{EGA-IV2} A.~Grothendieck  (r\'edig\'es avec la collaboration de Jean Dieudonn\'e),
\emph{\'El\'ements de g\'eom\'etrie alg\'ebrique : IV. \'Etude locale des sch\'emas et des morphismes de sch\'emas}, 
Deuxi\`eme partie Publ.~ math.~de l'I.H.\'E.S., tome {\bf 24} (1965), pp. 5--236.

\bibitem{EGA-IV3} A.~Grothendieck  (r\'edig\'es avec la collaboration de Jean Dieudonn\'e),
\emph{\'El\'ements de g\'eom\'etrie alg\'ebrique : IV. \'Etude locale des sch\'emas et des morphismes de sch\'emas}, 
Troisi\`eme partie Publ.~ math.~de l'I.H.\'E.S., tome {\bf 28} (1966), pp. 5--255.

\bibitem{Ho} M.~Hochster, \emph{Expanded radical ideals and semiregular ideals}, Pacific J.~Math. {\bf 44} (1973), pp.~553--568.

\bibitem{HH} M.~Hochster and C.~Huneke,  \emph{F-regularity, test elements, and smooth base change},
Trans. of the A.M.S., Vol. {\bf 346} (1994), pp.~1--62.

\bibitem{Ka} M.~Katzman, \emph{Parameter-test-ideals of Cohen-Macaulay rings}, Compos.~Math.~{\bf 144} (2008), pp.~933--948.

\bibitem{OR} J.~Ohm and  D.~Rush, \emph{Content modules and algebras}, Math.~Scand.~\textbf{31} (1972), pp.~49--68.

\bibitem{Pi} G.~Picavet, \emph{Absolutely integral homomorphisms},
J.~Algebra~\textbf{311} (2007), pp.~584--605.

\bibitem{Ru} D.~Rush, \emph{Content algebras}, Canad.~Math.~Bull.~\textbf{21} (1978), pp.~329--334.

\bibitem{Sh} R.~Y.~Sharp, \emph{Big tight closure test elements for some non-reduced excellent rings}, Journal of Algebra {\bf 349} (2012), pp.~284--316.

\end{thebibliography}

\end{document}